\documentclass[12pt, twoside]{article}
\usepackage{amsmath,amsthm,amssymb,color}
\usepackage[american]{babel}
\usepackage[title]{appendix}
\usepackage{bbm}
\usepackage{cases}
\usepackage{empheq}
\usepackage{enumerate}
\usepackage[T1]{fontenc}
\usepackage{float}
\usepackage{subfigure}
\usepackage[a4paper, left=2.5cm, right=2.5cm, top=2.5cm, bottom=2.5cm]{geometry}
\usepackage{graphicx}
\usepackage[colorlinks = true, urlcolor  = red,citecolor=blue]{hyperref}
\usepackage{indentfirst}
\usepackage{lineno}
\usepackage{microtype}
\usepackage{times}
\usepackage{url}
\usepackage[]{caption2} 
\renewcommand{\thesubfigure}{Fig. \arabic{subfigure}}
\makeatletter \renewcommand{\@thesubfigure}{\thesubfigure \space}
\renewcommand{\p@subfigure}{} \makeatother

\def\titlerunning#1{\gdef\titrun{#1}}
\makeatletter
\def\author#1{\gdef\autrun{\def\and{\unskip, }#1}\gdef\@author{#1}}
\def\address#1{{\def\and{\\\hspace*{18pt}}\renewcommand{\thefootnote}{}%
\footnote {#1}}%
\markboth{\autrun}{\titrun}}
\makeatother
\def\email#1{e-mail: #1}
\def\subjclass#1{{\renewcommand{\thefootnote}{}%
\footnote{\emph{Mathematics Subject Classification (2020):} #1}}}

\theoremstyle{plain}
\newtheorem{thm}{Theorem}[section]
\newtheorem{prop}[thm]{Proposition}
\newtheorem{lem}[thm]{Lemma}
\newtheorem{dfn}[thm]{Definition}
\newtheorem{cor}[thm]{Corollary}
\newtheorem{re}{Remark}
\newtheorem{ex}{Example}[section]
\numberwithin{equation}{section}

\newcommand{\eps}{\varepsilon}

\begin{document}
\baselineskip=16pt

\titlerunning{Quasi-stationary MFG}
\title{Existence of Solutions and Selection Problem for Quasi-stationary Contact Mean Field Games}
\author{Xiaotian Hu}
\date{}

\maketitle
\address{Xiaotian Hu : School of Mathematical Sciences, Shanghai Jiao Tong University, Shanghai 200240, China; \email{sjtumathhxt@sjtu.edu.cn}}

\begin{abstract}
 First, we study the existence of solutions for a class of first order mean field games systems
 \begin{equation*}
 	\left\{\begin{aligned}
 	&H(x,u,Du)=F(x,m(t)),\quad &&x\in M,\ \forall\ t\in[0,T],\\
 	&\partial_t m-\text{div}\left(m\dfrac{\partial H}{\partial p}(x,u,Du)\right)=0,\quad &&(x,t)\in M\times(0,T],\\
 	&m(0)=m_0,
\end{aligned}\right.
 \end{equation*}
 where the system comprises a stationary Hamilton-Jacobi equation in the contact case and an evolutionary continuity equation.
 
 Then, for any fixed $\lambda>0$, let $(u^\lambda,m^\lambda)$ be a solution of the system
 \begin{equation*}
 	\left\{
 	\begin{aligned}
 		&H(x,\lambda u^\lambda,Du^\lambda)=F(x,m^\lambda(t))+c(m^\lambda(t)),\quad &&x\in M,\ \forall t\in[0,T],\\
 		&\partial_t m^\lambda-\text{div}\left(m^\lambda\dfrac{\partial H}{\partial p}(x,\lambda u^\lambda,Du^\lambda)\right)=0,\quad &&(x,t)\in M\times(0,T],\\
 		&m(0)=m_0,
 	\end{aligned}\right.
 \end{equation*}
	 where $c(m^\lambda(t))$ is the Ma\~n\'e critical value of the Hamiltonian $H(x,0,p)-F(x,m^\lambda(t))$. We investigate the selection problem for the limit of $(u^\lambda,m^\lambda)$ as $\lambda$ tends to 0.

 \bigskip
 
\end{abstract}

\section{Introduction}
In 2007, Lasry, Lions \cite{LL07} introduced the mean field games (MFG) theory, which provides a mathematical framework to analyze decision-making in large population of small interacting agents. A standard first order system takes the following form
\begin{equation}\label{MFG-t}
	\left\{\begin{aligned}
		&-\partial_tu^T+H_0(x,Du^T)=f(x,m^T(t)),\quad &&(x,t)\in \mathbb{R}^d\times(0,T),\\
		&\partial_tm^T-\text{div}\left(m^T\dfrac{\partial H_0}{\partial p}(x,Du^T)\right)=0,\quad &&(x,t)\in\mathbb{R}^d\times(0,T),\\
		&u(x,T)=g(x,m(T)),\ m(0)=m_0,\quad &&x\in\mathbb{R}^d.
	\end{aligned}\right.
\end{equation}
This system consists of a Hamilton-Jacobi equation coupled with a continuity equation, where  $f$  represents a nonlocal coupling term. For any fixed $T>0$, $u^T$ and $m^T$ are unknowns, and the pair $(u^T,m^T)$ is a weak solution of the system \eqref{MFG-t}. Here, $u^T$ interprets as the value function of an agent, while $m^T$ describes the density of the agent population. The optimal strategy follows a feedback by $-D_p H_0(x,Du^T)$. When all agents adopt the same strategy, the system reaches an equilibrium.

However, in real-world scenarios, accurately predicting future states is often challenging for agents due to inherent uncertainties. Rather than solving a fully forward-looking optimization problem, an agent typically makes decisions based only on the information available at the current time  $t$. In particular, the agent observes the mean-field distribution $m(t)$ at that moment and determines an optimal strategy based on this present state, without explicitly accounting for its future evolution. To address this setting, Mouzouni \cite{M20} introduced the concept of quasi-stationary MFG systems, which captures decision-making under such conditions. Further developments on the quasi-stationary system can be found in \cite{CM23, CMM24arxiv}.

In this paper, we aim to investigate the quasi-stationary contact MFG of first order

	\begin{empheq}[left={(qMFG)}\qquad\empheqlbrace]{alignat=2}
		&H(x,u^T,Du^T)=F(x,m^T(t)),\quad &&x\in M,\ \forall\ t\in[0,T], \label{cH-J with m}\\
		&\partial_t m^T-\text{div}\left(m^T\dfrac{\partial H}{\partial p}(x,u^T,Du^T)\right)=0,\quad &&(x,t)\in M\times(0,T],\label{continuity eq}\\
		&m^T(0)=m_0, \label{initial value}
	\end{empheq}
where $M$ is a compact $d$-dimensional manifold without boundary  (e.g. $\mathbb{T}^d$), and $H:T^*M\times\mathbb{R}\to\mathbb{R}$ is a contact Hamiltonian. For any fixed $t\in[0,T]$, the equation \eqref{cH-J with m} is a stationary contact Hamilton-Jacobi equation with the nonlocal coupling term $F$, and $m(t)$ is a Borel probability measure on $M$ with its density, still denoted by $m(t)$. The equation \eqref{continuity eq} is an evolutionary continuity equation with the initial value \eqref{initial value}.  For simplicity, we denote $u^T$ and $m^T$ as $u$ and $m$, respectively.

Then, for any fixed $\lambda>0$, we consider the system
\begin{equation*}
	(qMFG_\lambda)\qquad
	\left\{
	\begin{aligned}
		&H(x,\lambda u^\lambda,Du^\lambda)=F(x,m^\lambda(t))+c(m^\lambda(t)),\quad &&x\in M,\ \forall t\in[0,T],\\
		&\partial_t m^\lambda-\text{div}\left(m^\lambda\dfrac{\partial H}{\partial p}(x,\lambda u^\lambda,Du^\lambda)\right)=0,\quad &&(x,t)\in M\times(0,T],\\
		&m(0)=m_0,
	\end{aligned}\right.
\end{equation*}
where $c(m^\lambda(t))$ is the Ma\~n\'e critical value of $H(x,0,p)-F(x,m^\lambda(t))$. For any fixed $\lambda>0$, there exists a weak solution $(u^\lambda,m^\lambda)$ of $(qMFG_\lambda)$. The selection problem aims to determine whether the limit of the family $(u^\lambda,m^\lambda$) as $\lambda \to 0$ exists and, if it does, to understand its characteristics and properties.

The selection problem was first explored in \cite{G08} using generalized Mather measures. In 2016, Davini, Fathi, Iturriaga and Zavidovique \cite{DFIZ16} studied the problem for the unique viscosity solution $w^\lambda$ of the discounted Hamilton-Jacobi equation
\begin{equation*}
	\lambda w^\lambda+h(x,Dw^\lambda)=c(h),\quad x\in M,
\end{equation*}
when $\lambda>0$ and $c(h)$ is the Ma\~n\'e critical value. They proved that as $\lambda\to 0$, the family of $w^\lambda$ uniformly converges to $w_0$, where $w_0$ is a viscosity solution of $h(x,Dw)=c(h)$ and $\lambda w^\lambda\to 0$. Since the publication of this result, a number of related works have appeared, advancing the study of the selection problem for the contact Hamiltonian $H(x,u,p)$, which is either increasing or decreasing in $u$ (see, for example, \cite{C23, CCIZ19, CFZZ24, DNYZ24arxiv, DW21, GMT18, WYZ21, Z22} and the references therein). The selection problem for second order mean field games can be found in \cite{CP19,CP21}. For the first order case, it remains an open problem, although similar results have been obtained for a special class of MFG systems \cite{GMT20,MT23}.

Our approach is based on weak KAM theory and some PDE techniques. Weak KAM theory, originally introduced by Fathi \cite{F08}, established a fundamental connection between viscosity solutions and the dynamics of positive-definite Hamiltonian systems. Cardaliaguet \cite{C13a} was among the first to incorporate weak KAM theory into the study of first order mean field games. For further applications of weak KAM methods in first order MFG, we refer to \cite{CCMW20, CCMW21, CMM24arxiv, IW23}. Hu and Wang \cite{HW22} investigated the existence of weak solutions to the stationary first order contact MFG by weak KAM theory for contact Hamiltonians \cite{IWWY21, WWY19a, WWY19b}.

The rest of the paper is organized as follows. In Section 2, we introduce the notations, and state our main results, as well as the assumptions on the contact Hamiltonian $H$, the nonlocal coupling term $F$ and the initial value $m_0$. In Section 3, we investigate the contact Hamilton-Jacobi equation, and review some weak KAM results, which are useful in the paper. In Section 4, we investigate the pushforward of the measure $m_0$, and prove the existence of weak solutions for $(qMFG)$ by fixed point theorem. In Section 5, we investigate the selection problem and prove the convergence result of the family of weak solutions to $(qMFG_\lambda)$.

\section{Preliminaries and main results}
Let $\mathcal{P}(M)$ denote the set of Borel probability measures on the state space $M$ with weak* topology. A sequence $\{m_n\}_{n\in\mathbb{N}}\subset \mathcal{P}(M)$ weakly* converges to $m\in\mathcal{P}(M)$, denoted by $m_n \xlongrightarrow{w^*}m$, if for any $\varphi\in C(M)$, we have
$$\lim_{n\to\infty}\int_M\varphi dm_n=\int_M \varphi dm.$$
 We recall the definition of Kantorovich-Rubinstein distance $d_1$ on probability measure spaces $\mathcal{P}_1(M)$, where $\mathcal{P}_1(M)$ is the set of Borel probability measures with the finite moment of order $1$ on $M$. More precisely, for any $m_1,m_2\in \mathcal{P}_1(M)$,
$$d_1(m_1,m_2)=\sup_{\varphi}\int_M \varphi d(m_1-m_2)=\inf_{\pi\in\Pi(\mu,\nu)}\int_{M\times M}\vert x-y\vert d\pi(x,y),$$
where the supremum is taken among all $1$-Lipschitz continuous functions $\varphi: M\to\mathbb{R}$, and $\Pi(\mu,\nu)$ is the set of Borel probability measures on $M\times M$ such that $\pi(A\times M)=\mu(A)$ and $\pi(M\times A)=\nu(A)$ for any Borel sets $A\subset M$. Note that 
$$d_1(m_n,m)\to 0\Longleftrightarrow m_n\xlongrightarrow{w^*}m.$$

Since $M$ is compact, we have that $\mathcal{P}_1(M)=\mathcal{P}(M)$. We denote by $C([0,T];\mathcal{P}(M))$ the set such that for any $m\in C([0,T];\mathcal{P}(M))$, $m(t)\in\mathcal{P}(M)$ and $m(\cdot)$ is continuous.

\begin{dfn}
	We call a pair $(u,m)\in C(M\times[0,T])\times C([0,T];\mathcal{P}(M))$ a weak solution of $(qMFG)$, if for any fixed $t\in[0,T]$, $u(\cdot,t)$ is a viscosity solution of \eqref{cH-J with m}, and $m$ is a distributional solution of \eqref{continuity eq} with the initial value \eqref{initial value}.
\end{dfn}

We assume that the Hamiltonian $H$ is of at least $C^3$ and satisfies the following assumptions:
\begin{enumerate}[(H1)]
	\item \textbf{Uniform convexity:} There is a constant $C_H>0$ such that for any $(x,u)\in M\times\mathbb{R}$, $$\dfrac{I}{C_H}\le\dfrac{\partial^2 H}{\partial p^2}(x,u,p)\le C_H I.$$
	\item \textbf{Superlinearity:} For any $(x,u)\in M\times\mathbb{R}$, $H(x,u,p)$ is superlinear in $p$, i.e.,
	$$\lim_{\vert p\vert\to\infty}\dfrac{H(x,u,p)}{\vert p\vert}=\infty,\quad \text{for any }(x,u)\in M\times\mathbb{R}.$$
	\item \textbf{Monotonicity:} There exist constants $\tau>0$, $\Lambda>0$ such that for all $(x,u,p)\in T^*M\times\mathbb{R}$,$$0<\tau<\dfrac{\partial H}{\partial u}(x,u,p)\le \Lambda.$$
\end{enumerate}
\begin{re}
	These assumptions on $H$ are stronger than classical Tonelli assumptions \cite{F08}. The $C^3$ regularity of $H$ is a technical assumption, as we need to use some weak KAM theory results, which are obtained in \cite{WWY19a}, on contact Hamiltonians.
\end{re}

	The contact Lagrangian $L:TM\times\mathbb{R}\to\mathbb{R}$ is defined as
	$$L(x,u,q):=\sup_{p\in T_x^*M}\{p\cdot q-H(x,u,p)\}.$$
	It is straightforward to verify that $L$ satisfies the following properties:
	\begin{enumerate}[(L1)]
		\item There is a constant $C_H>0$ such that $$\dfrac{I}{C_H}\le\dfrac{\partial^2 L}{\partial q^2}(x,u,q)\le C_H I.$$
		\item For any $(x,u)\in M\times\mathbb{R}$, $L(x,u,q)$ is superlinear in $q$, i.e.
		$$\lim_{\vert q\vert\to\infty}\dfrac{L(x,u,q)}{\vert q\vert}=\infty,\quad \text{for any}\ (x,u)\in M\times\mathbb{R}.$$
		\item There exist constants $\tau>0$, $\Lambda>0$ such that for all $(x,u,q)\in TM\times\mathbb{R}$,$$-\Lambda\le \dfrac{\partial L}{\partial u}(x,u,q)<-\tau<0.$$
	\end{enumerate}

\begin{ex}
	For any fixed $\beta>0$, the discounted Hamiltonian
	$$H(x,u,p)=\beta u+h(x,p)$$
	satisfies (H1)-(H3), if $h$ is at least of $C^3$ and satisfies (H1)-(H2). The discounted Lagrangian is denoted by $L(x,u,q)=-\beta u+l(x,q)$, where $l(x,q)$ is the Lagrangian associated to the Hamiltonian $h$.
\end{ex}

The nonlocal coupling term $F:M\times\mathcal{P}(M)\to\mathbb{R}$ satisfies
\begin{enumerate}[(F1)]
	\item The map $x\mapsto F(x,m)$ is of class $C^2$ with $$\sup_{m\in\mathcal{P}(M)}\sum_{\vert \alpha\vert \le 2}\Vert D^\alpha F(\cdot,m)\Vert_\infty<\infty.$$
	\item The function $F$ and its derivative $D_xF$ are both continuous on $M\times\mathcal{P}(M)$.
	\item The map $m\mapsto F(x,m)$ is Lipschitz continuous with Lipschitz constant $C_F$, i.e.,
	$$\vert F(x,m_1)-F(x,m_2)\vert\le C_F d_1(m_1,m_2),\quad \text{for any }m_1,m_2\in\mathcal{P}(M),x\in M.$$
\end{enumerate}

Then we impose an assumption on the initial value $m_0$ in \eqref{initial value}.
\begin{enumerate}[(P)]
	\item The Borel probability measure $m_0$ is absolutely continuous with respect to Lebesgue measure $\mathcal{L}$ with the density, still denoted by $m_0$, which is bounded.
\end{enumerate}
The assumption (P) guarantees that the pushforward of the measure $m_0$ is also absolutely continuous and bounded (see Lemma \ref{4 push-for ac bound}). In this paper, we always assume (H1)-(H3), (F1)-(F3) and (P).

We present the first main result.
\begin{thm}\label{main result 1}
	Assume (H1)-(H3), (F1)-(F3) and (P). There exists a weak solution $(\bar{u},\bar{m})\in C(M\times[0,T])\times C([0,T];\mathcal{P}(M))$ of $(qMFG)$ such that
	\begin{enumerate}[(I)]
		\item  $\bar{u}(\cdot,t)$ is a viscosity solution of \eqref{cH-J with m} for any fixed $t\in[0,T]$, and $\bar{m}$ is a distributional solution of \eqref{continuity eq}.
		\item  $\bar{m}$ is Lipschitz continuous on $[0,T]$, i.e., there exists some constant $C\ge0$, such that for any $t,s,\in[0,T]$, we have
		$$d_1(\bar{m}(t),\bar{m}(s))\le C\vert t-s\vert.$$ 
		
		For any $t\in[0,T]$, $\bar{m}(t)$ is absolutely continuous with respect to the Lebesgue measure $\mathcal{L}$, and that we denote by $\bar{m}(t)$ also the density, which satisfies $\Vert\bar{m}(t)\Vert_\infty\le C$ for some constant $C=C(T,\Vert m_0\Vert_\infty,H)$.
		\item $\bar{u}(\cdot,t)$ is Lipschitz continuous and semi-concave uniformly on $[0,T]$ and $\bar{u}(x,\cdot)$ is Lipschitz continuous uniformly on $M$.
	\end{enumerate}
\end{thm}

\begin{re}
	\begin{enumerate}[(i)]
		\item In \cite{HW22}, Hu and Wang proved the existence of weak solutions to the stationary contact MFG system
		\begin{equation}\label{2 stationary MFG}
			\left\{\begin{aligned}
				&H(x,u,Du)=F(x,m),\quad &&x\in M,\\
				&\text{div}\left(m\dfrac{\partial H}{\partial p}(x,u,p)\right)=0,\quad &&x\in M,\\
				&\int_M m=1.
			\end{aligned}\right.
		\end{equation}
		under the additional assumption that
		\begin{enumerate}[(R1)]
			\item for any $(x,u,p)\in T^*M\times\mathbb{R}$, $H(x,u,p)=H(x,u,-p).$
		\end{enumerate}
		Let $(u,m)\in C(M)\times\mathcal{P}(M)$ be a weak solution of the system \eqref{2 stationary MFG}. So $u=u_m$ is the unique viscosity solution to the Hamilton-Jacobi equation in \eqref{2 stationary MFG}, and there exists a Mather measure $\mu_m$ for the Hamiltonian $H(x,u,p)-F(x,m)$ such that $m=\pi_x\sharp\mu_m$, where $\pi_x:T^*M\times\mathbb{R}\to M$ denotes the canonical projection. Note that, under (R1), the Mather set \cite{WWY19a} associated to the Hamiltonian $H(x,u,p)-F(x,m)$ is
		$$\tilde{M}=\{(x,u_m(x),0)\vert\ H(x,u_m(x),0)=F(x,m)\}.$$
		Mather measures are given by the convex combination of Dirac measures $\delta_{\{x,u_m(x),0\}}$.
		
		However, in this paper, the assumption (R1) is not required, as the analysis of Mather measures is unnecessary and the weak solution $m(t)$ evolves with the flow \eqref{4 dfn of flow}.Even in the absence of (R1), the uniform boundedness, equi-Lipschitz continuity and the uniform semi-concavity of $\{u_{m(t)}\}$ stiil hold (see Proposition \ref{3 u_m bound Lip}).

		\item In \cite{CMM24arxiv}, Camilli, Marchi and Mendico proved the existence of weak solutions to the quasi-stationary MFG system
		\begin{equation}\label{2 classical qMFG}
			\left\{
			\begin{aligned}
				&H_0(x,Du,m(t))=c(m(t)),\quad &&x\in M,\ \forall t\in[0,T],\\
				&\partial_t m-\text{div}\left(m\dfrac{\partial H_0}{\partial p}(x,Du,m)\right)=0,\quad &&(x,t)\in M\times(0,T],\\
				&m(0)=m_0,
			\end{aligned}\right.
		\end{equation}
		where $H_0$ is a non-separable Hamiltonian, which means that the Hamiltonian is dependent on $m$, and $c(m(t))$ is the Ma\~n\'e critical value \cite{M97} of $H_0$, under the additional assumption that
		\begin{enumerate}[(R2)]
			\item for any $m\in C([0,T];\mathcal{P}(M))$, there exists a unique $x_m\in M$ such that
			$$x_m\in\bigcap_{t\in[0,T]}\mathcal{A}_0^{m(t)},$$
			where for any fixed $t\in [0,T]$, $\mathcal{A}_0^{m(t)}$ is the projected Aubry set (see \cite{F08}) of the Hamiltonian $H_0(x,p,m(t))$.
		\end{enumerate}
		In \cite{CMM24arxiv}, a weak solution to the system is constructed as a fixed point of the following map. Given a flow of absolutely continuous measures $m(t)$, one considered $h_{m(t)}(x_m,\cdot)$, where $x_m$ is as in assumption (R2) and $h_{m(t)}$ is the Peierls barrier associated to the Hamiltonian $H_0(x,p,m(t))$, and to it associated, the pushforward of $m_0$ with respect to the flow with drift $Dh_{m(t)}(x_m,\cdot)$.
		
		In our setting, due to assumption $(H3)$, given the flow of measures $m(t)$, the Hamilton-Jacobi equation admits a unique viscosity solution $u_{m(t)}$, whose properties are recalled in Proposition \ref{3 u_m bound Lip}, Proposition \ref{3 u lip in t}.
	\end{enumerate}
\end{re}

Subsequently, we consider the system
\begin{equation*}
	(qMFG_0)\qquad \left\{\begin{aligned}
		&H(x,0,Dv)=F(x,\mu(t))+c(\mu(t)),\quad &&x\in M,\ \forall t\in[0,T],\\
		&\partial_t\mu-\text{div}\left(\mu\dfrac{\partial H}{\partial p}(x,0,Dv)\right)=0,\quad &&(x,t)\in M\times(0,T],\\
		&\mu(0)=m_0,
	\end{aligned}\right.
\end{equation*}
where the existence result was proved in \cite{CMM24arxiv} under assumption (R2). Let $(\bar{v},\bar{\mu})$ be the weak solution of $(qMFG_0)$. We aim to establish the selection criterion for the limit of the family $(u^\lambda,m^\lambda)$, and identify the connection between the limit and $(\bar{v},\bar{\mu})$. To proceed. we need an additional assumption:
\begin{enumerate}[(H4)]
	\item For any $m\in C([0,T];\mathcal{P}(M))$, there exists a unique point $x_m\in M$ such that 
	$$\mathcal{A}^{m(t)}=\{x_m\},\ \text{for any}\ t\in[0,T],$$
	where $\mathcal{A}^m$ is the projected Aubry set of $H(x,0,p)-F(x,m)$.
\end{enumerate}

\begin{re}
	Note that, under (H4), for any $m\in\mathcal{P}(M)$, the equation $H(x,0,Dv)=F(x,m)$ admits a unique viscosity solution up to additive constants.
\end{re}

\begin{re}
	It is clear that the assumption (H4) is stronger than (R2). Nonetheless, all the examples presented in \cite{CMM24arxiv} fit our assumption (H4). We recall them briefly. Let $H(x,0,p)=\dfrac{\vert p\vert^2}{2}-F(x,m)$, where $F$ is defined on $M\times\mathcal{P}(M)$, satisfying (F1)-(F3), and for any $m\in\mathcal{P}(M)$, 
	$$F(0,m)=0,\quad F(x,m)>0,\ \forall x\in M\setminus\{0\}.$$
	For example, given a function $\kappa\in C^2(M\times M)$ with $\kappa(0,y)\equiv0$ and $\kappa(x,y)>0$ in $M\setminus\{0\}\times M$, the function
	$$F(x,m)=\int_M \kappa(x,y)dm(y).$$
	In this case, we have
		$$c(m)=0,\ \mathcal{A}^m=\{0\},\quad \forall m\in\mathcal{P}(M).$$
\end{re}

In the system $(qMFG_\lambda)$, the contact Hamilton-Jacobi equation with an additional term $c(m)$
$$H(x,u,Du)=F(x,m)+c(m),\quad x\in M,$$
 admits a unique viscosity solution \cite[Theorem B.1]{WWY19a} under a weaker monotonicity assumption:
\begin{itemize}
	\item[(H3')] There exists a constant $\Lambda>0$ such that for all $(x,u,p)\in T^*M\times\mathbb{R}$,$$0<\dfrac{\partial H}{\partial u}(x,u,p)\le \Lambda.$$
\end{itemize}
Under (H3'), the existence result for $(qMFG_\lambda)$ still holds. See the proof in the Appendix.

Finally, we present the following selection criterion for the limit, in which we construct a weak solution of $(qMFG_0)$ in a way that differs from the approach \cite{CMM24arxiv}.
\begin{thm}\label{main result 2}
	Assume (H1), (H2), (H3'), (H4), (F1)-(F3) and (P). For any $\lambda>0$, denote by $(u^\lambda,m^\lambda)\in C(M\times[0,T])\times C([0,T];\mathcal{P}(M))$ a weak solution of $(qMFG_\lambda)$. Then, for any $\bar{x}\in M$, as $\lambda\to0$, the family $u^\lambda-u^\lambda(\bar{x},\cdot)$ uniformly converges (up to a subsequence) to the function $\bar{v}-\bar{v}(\bar{x},\cdot)\in C(M\times[0,T])$, and the family $m^\lambda$ (up to a subsequence) satisfies that 
	$$\lim_{\lambda\to0}\sup_{t\in[0,T]}d_1(m^\lambda(t),\bar{\mu}(t))=0,$$ for some $\bar{\mu}\in C([0,T];\mathcal{P}(M))$.
	Moreover, $(\bar{v},\bar{\mu})$ is a weak solution of $(qMFG_0)$ with Ma\~n\'e critical value $c(\bar{\mu})$, and $$\bar{v}-\bar{v}(\bar{x},\cdot)=h_{\bar{\mu}(\cdot)}(x_{\bar{\mu}},\cdot)-h_{\bar{\mu}(\cdot)}(x_{\bar{\mu}},\bar{x}).$$
\end{thm}

Although the Hamilton-Jacobi equation in $(qMFG_\lambda)$ is stationary, the equi-continuity result of $u(x,\cdot)$ is crucial in proving the convergence (see Proposition \ref{5 regular of v lambda}). As $\lambda\to 0$, we do not have the sufficient information on the convergence of the family $\{u^\lambda\}$, but we investigate the convergence of the family $\{u^\lambda-u^\lambda(\bar{x},\cdot)\}$.

\section{Contact Hamilton-Jacobi equations}
In this section, we consider the contact Hamilton-Jacobi equation 
\begin{equation}\label{cH-J without m}
	H(x,u,Du)=c,\quad x\in M,
\end{equation} for some constant $c\in\mathbb{R}$, and review some weak KAM results for the contact Hamiltonian system.

\begin{dfn}
	A continuous function $u:M\to \mathbb{R}$ is a viscosity subsolution of \eqref{cH-J without m}, if for any $\varphi\in C^1(M)$ and any points $y\in M$ such that $u-\varphi$ attains the local maximum at $y$, we have
	$$H(y,u(y),D\varphi(y))\le c.$$
	
	A continuous function $u:M\to \mathbb{R}$ is a viscosity supersolution of \eqref{cH-J without m}, if for any $\varphi\in C^1(M)$ and any points $y\in M$ such that $u-\varphi$ attains the local minimum at $y$, we have
	$$H(y,u(y),D\varphi(y))\ge c.$$
	
	A continuous function $u:M\to \mathbb{R}$ is a viscosity solution of \eqref{cH-J without m}, if $u$ is both a viscosity subsolution and a viscosity supersolution of \eqref{cH-J without m}.
\end{dfn}

It is well-known that there exists a viscosity solution to the classical Hamilton-Jacobi equation
\begin{equation}\label{H-J}
	H_0(x,Du)=c,\quad x\in M,
\end{equation} 
if and only if the constant $c$ in \eqref{H-J} is equal to the Ma\~n\'e critical value $c(H_0)$, as introduced by Ma\~n\'e \cite{M97}. There are some other formulas of $c(H_0)$. In \cite{CIPP98}, we have $$c(H_0)=\inf_{\varphi\in C^1(M)}\max_{x\in M}H_0(x,D\varphi(x)).$$ In \cite{F08}, it is given by
$$c(H_0)=-\inf_\mu\int_{TM}L_0(x,q)d\mu,$$
where $L_0:TM\to\mathbb{R}$ is the Lagrangian associated to the Hamiltonian $H_0$, and the infimum is taken among Borel probability measures on $TM$ invariant under the Euler-Lagrange flow.

On the other hand, there exists more than one constant that makes the equation \eqref{H-J} admit a viscosity solution. We denote the set of these constants by $\mathcal{G}$, and the construction of $\mathcal{G}$ is dependent on the monotonicity of $H$ with respect to $u$. For results when $-\Lambda\le\frac{\partial H}{\partial u}\le\Lambda$, see \cite{WY21arxiv,WWY19b}. If $0\le\frac{\partial H}{\partial u}\le\Lambda$, then we have $$\mathcal{G}=\{c(H(x,a,p))\vert a\in\mathbb{R}\},$$
where $c(H(x,a,p))$ is the Ma\~n\'e critical value of $H(x,a,p)$ \cite{SWY16, H25}. Under (H1)-(H3), we have $\mathcal{G}=\mathbb{R}$ \cite[Remark 1.2]{WWY19a}, which means that for any $c\in\mathbb{R}$, the equation \eqref{cH-J without m} admits a viscosity solution, and the viscosity solution is unique \cite[Proposition A.1]{WWY19a}. Moreover, the unique viscosity solution is Lipschitz continuous on $M$.

Given $m\in C([0,T];\mathcal{P}(M))$, consider the equation \eqref{cH-J with m} $$H(x,u,Du)=F(x,m(t)),\quad\ x\in M,\ \forall t\in[0,T].$$ By the uniqueness of viscosity solutions, for any fixed $t\in[0,T]$, the equation \eqref{cH-J with m} admits a unique viscosity solution denoted by $u_{m(t)}$. According to \cite[Appendix B]{WWY19a}, for all $m(t)\in\mathcal{P}(M)$, there is a constant $a_{m(t)}$ such that 
$$\inf_{\varphi\in C^1(M)}\max_{x\in M}(H(x,a_{m(t)},D\varphi)-F(x,m(t)))=0.$$
$a_{m(t)}$ is uniformly bounded for any $m\in C([0,T];\mathcal{P}(M))$ and any $t\in[0,T]$ \cite[Lemma 1]{HW22}.

Next, we establish the boundedness and the regularity of $u_{m(t)}$.

\begin{prop}\label{3 u_m bound Lip}
     The family $\{u_{m(t)}\}_{m\in C([0,T];\mathcal{P}(M))}$ is uniformly bounded, equi-Lipschitz continuous and uniformly semi-concave on $M$.
\end{prop}
\begin{proof}[Outline of proof]
	For simplicity, we prove that for $m\in \mathcal{P}(M)$, the viscosity solution $u_m$ is uniformly bounded, equi-Lipschitz continuous and uniformly semi-concave on $M$, uniformly with respect to $m$.
	
	To prove the uniform boundedness, by the aforementioned arguments, for any $m\in\mathcal{P}(M)$, there exists a constant $a_m\in\mathbb{R}$ such that 
	$$\inf_{\varphi\in C^1(M)}\max_{x\in M}(H(x,a_m,D\varphi)-F(x,m))=0,$$
	and $a_m$ is uniformly bounded \cite[Lemma 1]{HW22}. Then we get the uniform boundedness of $u_m$ by utilizing the Lax-Oleinik semigroup of the Lagrangian $L(x,a_m,q)-F(x,m)$ \cite[Proposition 9]{HW22}.
	
	The equi-Lipschitz continuity is the consequence of \cite[Lemma 5]{HW22}, and the uniform semi-concavity follows from \cite[Theorem 3.3]{L82}.
\end{proof}

In the following proposition, we establish the regularity of $u_{m(\cdot)}$.
\begin{prop}\label{3 u lip in t}
	For any $m\in C([0,T];\mathcal{P}(M))$, the viscosity solution $u_{m(\cdot)}$ is continuous on $[0,T]$. Especially, if $m$ is Lipschitz continuous on $[0,T]$, then $u_{m(\cdot)}$ is Lipschitz continuous on $[0,T]$.
\end{prop}
\begin{proof}
	For any $t,s\in[0,T]$, we consider viscosity solutions $u_{m(t)}$ and $u_{m(s)}$ of the Hamilton-Jacobi equation \eqref{cH-J with m} with fixed time $t$ and $s$, respectively. Define
	\begin{equation*}
		\begin{aligned}
			&v^+:=u_{m(s)}+\dfrac{1}{\tau}\Vert F(\cdot,m(t))-F(\cdot,m(s))\Vert_\infty,\\
			&v^-:=u_{m(s)}-\dfrac{1}{\tau}\Vert F(\cdot,m(t))-F(\cdot,m(s))\Vert_\infty.
		\end{aligned}
	\end{equation*}
	Then we have
	\begin{equation*}
		\begin{aligned}
			&H(x,v^+,Dv^+)\ge H(x,u_{m(s)},Du_{m(s)})+\Vert F(\cdot,m(t))-F(\cdot,m(s))\Vert_\infty\ge F(x,m(t)),\\
			&H(x,v^-,Dv^-)\le H(x,u_{m(s)},Du_{m(s)})-\Vert F(\cdot,m(t))-F(\cdot,m(s))\Vert_\infty\le F(x,m(t)).
		\end{aligned}
	\end{equation*}
	So $v^+$ and $v^-$ are respectively a viscosity supersolution and a viscosity subsolution of 
	$$H(x,u,Du)=F(x,m(t)),\quad x\in M.$$
	By comparison principle and assumption (F3), we have
	\begin{equation}\label{3 u Lip m Lip}
		\Vert u_{m(t)}(\cdot)-u_{m(s)}(\cdot)\Vert_\infty\le \dfrac{C_F}{\tau}d_1(m(t),m(s)),
	\end{equation}
which implies that $u_{m(\cdot)}$ is continuous in $t$.

If $m$ is Lipschitz continuous on $[0,T]$,	by the same argument and inequality \eqref{3 u Lip m Lip}, then $u_{m(\cdot)}$ is Lipschitz on $[0,T]$.
\end{proof}

Under (H3'), we only obtain the continuity of $u_{m(\cdot)}$, not Lipschitz continuity (see Proposition \ref{AA continuity in t}).

\begin{cor}\label{3 Du measurable in t}
	 The set-valued map $t\mapsto D^+u_{m(t)}(x)$ is measurable for all $x\in M$, where $D^+u(x,t)$ is the superdifferential of $u$ in $x$.
\end{cor}
\begin{proof}
	The corollary is a standard consequence of Proposition \ref{3 u lip in t} and the semi-concavity of $u$ in $x$ \cite[Proposition 3.3.4]{CS04}.
\end{proof}

\section{Existence result}
In this section, we study the continuity equation and prove Theorem \ref{main result 1}.

From the regularities of $u_m$ and the measurability of $D^+u_m$ in $t$ as proved in Section 3, with \cite[Theorem 2.3.1]{CS04}, the function
$$x\mapsto \dfrac{\partial H}{\partial p}(x,u_{m(t)}(x),Du_{m(t)}(x))$$ 
is of bounded variation on $M$ for any fixed $t\in[0,T]$. Then, \cite[Remark 1.2 and Theorem 3.11]{AC14} ensures that, given $m\in C([0,T];\mathcal{P}(M))$, for $t\in [0,T]$ and $x\in M$, we define the flow 
\begin{equation}\label{4 dfn of flow}
	\Phi_m(x,t)=x-\int_{0}^{t}\dfrac{\partial H}{\partial p}(\Phi_m(x,s),u_{m(s)}(\Phi_m(x,s)),Du_{m(s)}(\Phi_m(x,s)))ds,
\end{equation}
where $u_{m(s)}$ is the viscosity solution of $$H(x,u,Du)=F(x,m(s)),\quad x\in M.$$

\begin{lem}\label{4 flow Lip t}
	For any $m\in C([0,T];\mathcal{P}(M))$, there is a constant $C_1\ge 0$, independent of $m$, such that 
	$$\vert \Phi_m(x,t)-\Phi_m(x,t')\vert \le C_1\vert t-t'\vert,\qquad \forall\ x\in M,\ \forall t,t'\in[0,T].$$
\end{lem}
\begin{proof}
	By Proposition \ref{3 u_m bound Lip} and the compactness of $M$, there exists a constant $C_1\ge0$ such that for any $m\in C([0,T];\mathcal{P}(M))$, we have 
	\begin{equation*}
		\begin{aligned}
			&\vert \Phi_m(x,t)-\Phi_m(x,t')\vert\\
			\le &\left\vert\int_{t'}^{t}\dfrac{\partial H}{\partial p}(\Phi_m(x,s),u_{m(s)}(\Phi_m(x,s)),Du_{m(s)}(\Phi_m(x,s)))ds\right\vert\\
			 \le &C_1\vert t-t'\vert,\qquad \forall\ x\in M,\ \forall t,t'\in [0,T],
		\end{aligned}
	\end{equation*}
	where the second inequality is due to Proposition \ref{3 u_m bound Lip}. The proof is complete.
\end{proof}

\begin{lem}\label{4 flow inverse Lip x}
	For any $m\in C([0,T];\mathcal{P}(M))$, there exists a constant $C_2\ge 0$, independent of $m$, such that 
	$$\vert x-y\vert \le C_2\vert \Phi_m(x,t)-\Phi_m(y,t)\vert,\qquad \forall t\in[0,T],\ \forall x,y\in M.$$
\end{lem}
\begin{proof}
	The proof follows a similar argument as \cite[Lemma 4.3]{C13b}. For any $t\in[0,T]$, define the backward flow $$x(r):=\Phi_m(x,t-r),\ y(r):=\Phi_m(y,t-r),\ \forall r\in[0,t]$$ Then we have 
	\begin{equation*}
		\begin{aligned}
				\dot{x}(r)=\dfrac{\partial H}{\partial p}\left(\Phi_m(x,t-r),u_{m(t-r)}(\Phi_m(x,t-r)),Du_{m(t-r)}(\Phi_m(x,t-r))\right),\\
				\dot{y}(r)=\dfrac{\partial H}{\partial p}\left(\Phi_m(y,t-r),u_{m(t-r)}(\Phi_m(y,t-r)),Du_{m(t-r)}(\Phi_m(y,t-r))\right).
		\end{aligned}
\end{equation*}
Thus, the difference of $\dot{x}(r)$ and $\dot{y}(r)$ satisfies
\begin{equation*}
	\begin{aligned}
		&\dot{x}(r)-\dot{y}(r)\\
		=&\dfrac{\partial H}{\partial p}\left(\Phi_m(x,t-r),u_{m(t-r)}(\Phi_m(x,t-r)),Du_{m(t-r)}(\Phi_m(x,t-r))\right)\\
		&-\dfrac{\partial H}{\partial p}\left(\Phi_m(y,t-r),u_{m(t-r)}(\Phi_m(y,t-r)),Du_{m(t-r)}(\Phi_m(y,t-r))\right)\\
		\le& \tilde{C}_2\vert x(r)-y(r)\vert,
	\end{aligned}
\end{equation*}
for some constant $\tilde{C}_2\ge 0$, where the inequality is due to the locally Lipschitz continuity of $H$ and Proposition \ref{3 u_m bound Lip}.  For a.e. $r\in[0,t]$,
$$\dfrac{d}{dr}\left(\dfrac{1}{2}\vert x(r)-y(r)\vert^2 \right)=\left\langle \dot{x}(r)-\dot{y}(r),x(r)-y(r)\right\rangle\le \tilde{C}_2\vert x(r)-y(r)\vert^2.$$
Applying Grönwall inequality, we obtain
$$\vert x(0)-y(0)\vert\ge e^{-\tilde{C}_2r}\vert x(r)-y(r)\vert,\quad \forall r\in[0,t].$$
Let $C_2:=e^{\tilde{C}_2T}$ and the proof is complete.
\end{proof}

For any fixed $m\in C([0,T];\mathcal{P}(M))$, we define the pushforward of the measure $m_0$ by the flow \eqref{4 dfn of flow}, $$\mu(t)(A):=\Phi_m(\cdot,t)\sharp m_0(A)=m_0(\Phi^{-1}_m(\cdot,t)(A))=m_0\{x\vert\Phi_m(x,t)\in A\},\quad \forall t\in[0,T],$$ for any Borel sets $A\subset M$. Next, we establish some properties of the pushforward.

\begin{lem}\label{4 push-for ac bound}
	For any $m\in C([0,T];\mathcal{P}(M))$, there exists a constant $C_3\ge 0$, where $C_3$ only depends on $T$, $\Vert m_0\Vert_\infty$ and the Hamiltonian $H$, such that the pushforward $\mu$ is Lipschitz continuous on $[0,T]$ with respect to $d_1$ distance, i.e.
	$$d_1(\mu(t),\mu(s))\le C_3\vert t-s\vert,\quad \forall t,s\in[0,T].$$
	Moreover, for any $t\in[0,T]$, the measure $\mu(t)$ is absolutely continuous w.r.t. Lebesgue measure $\mathcal{L}$ and its density, still denoted by $\mu(t)$, satisfies $$\Vert \mu(t)\Vert_\infty\le C_3.$$
\end{lem}
\begin{proof}
	We first prove the Lipschitz continuity of $\mu$. For any $t,s\in[0,T]$, any $1$-Lipschitz continuous functions $\varphi$, we have 
	\begin{equation*}
		\begin{aligned}
			&d_1(\mu(t),\mu(s))\le\int_M \varphi(x)d(\mu(t)-\mu(s))=\int_M \varphi(\Phi_m(x,t))-\varphi(\Phi_m(x,s))dm_0\\
			\le&\int_M\Vert D\varphi\Vert_\infty\int_{s}^{t}-\dfrac{\partial H}{\partial p}\left(\Phi_m(x,r),u_{m(r)}(\Phi_m(x,r)), Du_{m(r)}(\Phi_m(x,r))\right)dr dm_0\\
			\le &C'_3\vert t-s\vert.
		\end{aligned}
	\end{equation*}
	Then, for any $s\in [0,T]$, for any Borel set $A\subset M$, we have
	\begin{equation*}
		\begin{aligned}
			\mu(s)(A)&=m_0(\Phi^{-1}_m(\cdot,s)(A))=m_0\{x\vert \Phi_m(x,s)\in A\}\\
			&=m_0\{\Psi_m(A,s)\}\le\Vert m_0\Vert_\infty\mathcal{L}(\Psi_m(A))\le C_2\Vert m_0\Vert_\infty\mathcal{L}(A),
		\end{aligned}
	\end{equation*}
	where the last inequality follows from Lemma \ref{4 flow inverse Lip x} and $\Psi$ denotes the inverse of $\Phi$.  Finally, we obtain
	$$\Vert \mu(t)\Vert_\infty\le C_2\Vert m_0\Vert_\infty,\quad \forall t\in[0,T].$$ Let $C_3:=\max\{C_2\Vert m_0\Vert_\infty, C'_3\}$ and the proof is complete.
\end{proof}

We consider the continuity equation
\begin{equation}\label{continuity eq with initial}
	\left\{\begin{aligned}
		&\partial_t m-\text{div}\left(m\dfrac{\partial H}{\partial p}(x,u,Du)\right)=0,\quad (x,t)\in M\times(0,T],\\
		&m(0)=m_0.
	\end{aligned}\right.
\end{equation}

\begin{prop}\label{4 sol of continuity eq}
	The map $s\mapsto \mu(s):=\Phi_m(\cdot,s)\sharp m_0$ is the unique weak solution of \eqref{continuity eq with initial}.
\end{prop}
\begin{proof}
	It is clear that $\mu(0)=m_0$. For any text function $\varphi\in C^\infty(M\times[0,T])$, we have 
	\begin{equation*}
		\begin{aligned}
			&\dfrac{d}{dt}\int_M\varphi(x,t)d\mu(t)=\dfrac{d}{dt}\int_M\varphi(\Phi_m(x,t),t)dm_0\\
			=&\int_M \partial_t\varphi(\Phi_m(x,t),t)\\
			&-\left\langle \dfrac{\partial H}{\partial p}\left(\Phi_m(x,t),u_{m(t)}(\Phi_m(x,t)),Du_{m(t)}(\Phi_m(x,t))\right), D\varphi(\Phi_m(x,t),t)\right\rangle dm_0\\
			&=\int_M\partial_t\varphi(x,t)-\left\langle \dfrac{\partial H}{\partial p}\left(x,u_{m(t)}(x),Du_{m(t)}(x)\right),D\varphi(x,t)\right\rangle d\mu(t).
		\end{aligned}
	\end{equation*}
Integrating over $[0,T]$ implies
$$\int_{0}^{T}\int_M\partial_t\varphi(x,t)-\left\langle \dfrac{\partial H}{\partial p}\left(x,u_{m(t)}(x),Du_{m(t)}(x)\right),D\varphi(x,t)\right\rangle d\mu(t)dt=0.$$
This shows that $\mu$ is a weak solution of equation \eqref{continuity eq with initial}, and the uniqueness follows from the results in \cite{AC14}.
\end{proof}

We are now in the position to prove the main result, Theorem \ref{main result 1}.
\begin{proof}[Proof of Theorem \ref{main result 1}]
      We define the set
      $$\mathcal{D}:=\left\{m\in C([0,T]; \mathcal{P}(M))\ \bigg\vert\ \sup_{t\neq s}\dfrac{d_1(m(t),m(s))}{\vert t-s\vert}\le C_3\right\},$$ where $C_3$ is as in Lemma \ref{4 push-for ac bound},
      and the map
      $$\mathcal{S}:\mathcal{D}\to\mathcal{D},$$ where $\mathcal{S}(m)=\mu$ and $\mu(t):=\Phi_m(\cdot,t)\sharp m_0$, $\forall t\in[0,T]$.
      
      We consider a sequence $\{m_n\}_{n\in\mathbb{N}}\subset \mathcal{D}$ such that there exists $m\in\mathcal{D}$ with $$\sup_{t\in[0,T]}d_1(m_n(t),m(t))\to 0.$$ Note that $F(\cdot,m_n(t))\to F(\cdot,m(t))$ uniformly on $[0,T]$.
      Let $\{u_{m_n(t)}\}_{n\in\mathbb{N},t\in[0,T]}$ be a family of continuous functions such that for any fixed $t$, $u_{m_n(t)}$ is a viscosity solution of $H(x,u,Du)=F(x,m_n(t)),\ \text{on}\ M$. Since $\{u_{m_n(t)}\}_{n\in\mathbb{N},t\in[0,T]}$ is uniformly bounded and equi-Lipschitz continuous, by the stability and uniqueness of viscosity solutions, we conclude that for any $t\in[0,T]$, $u_{m_n(t)}$ uniformly converges to $u_{m(t)}$, which is the viscosity solution of $H(x,u,Du)=F(x,m(t)).$
      
      Let $\{\mu_n\}_{n\in\mathbb{N}}$ be the sequence of pushforwards, which means that $\mu_n(t)=\Phi_{m_n}(\cdot,t)\sharp m_0$, for any $n\in\mathbb{N}$, and
      $$\Phi_{m_n}(x,t)=x-\int_{0}^{t}\dfrac{\partial H}{\partial p}\left(\Phi_{m_n}(x,s),u_{m_n(s)}(\Phi_{m_n}(x,s)),Du_{m_n(s)}(\Phi_{m_n}(x,s))\right)ds.$$
      By the uniform semi-concavity of $\{u_{m_n(t)}\}_{n\in\mathbb{N},t\in[0,T]}$, $Du_{m_n(t)}$ converges to $Du_{m(t)}$ a.e. on $M\times[0,T]$. Hence, for any $f\in C(M)$, we have for any $t\in[0,T]$,
      \begin{equation*}
      	\begin{aligned}
      		&\lim_{n\to\infty}\int_M f(x)d\mu_n(t)=\lim_{n\to\infty}\int_M f(\Phi_{m_n}(x,t))dm_0\\
      		=&\int_M f(\Phi_m(x,t))dm_0=\int_M f(x)d\mu(t),
      	\end{aligned}
      \end{equation*}
      which implies that 
      $$\mu_n(t)\xlongrightarrow{w^*}\mu(t)=\Phi_m(\cdot,t)\sharp m_0\in\mathcal{P}(M),\ \text{for any}\ t\in[0,T],$$
      with $\mu\in\mathcal{D}$. Thus, $\mu$ is a weak solution of \eqref{continuity eq with initial}. Thanks to Schauder fixed point theorem, there exists a fixed point $\bar{m}\in\mathcal{D}$ such that $\bar{m}=\Phi_{\bar{m}}\sharp m_0$ and the function $\tilde{u}_{\bar{m}}\in C(M\times[0,T])$, denoted by $\bar{u}(x,t):=\tilde{u}_{\bar{m}(t)}(x)$ for any $(x,t)\in M\times[0,T]$, is the viscosity solution of $$H(x,\bar{u},D\bar{u})=F(x,\bar{m}(t)),\quad x\in M,\ \forall t\in[0,T].$$ The proof of $(I)$ is complete.
      
      The results in $(II)$ are direct consequences of Lemma \ref{4 push-for ac bound}. The first result in $(III)$ follows as a corollary of Proposition \ref{3 u_m bound Lip}. The Lipschitz continuity of $\bar{u}(x,\cdot)$ is due to Proposition \ref{3 u lip in t} and the Lipschitz continuity of $\bar{m}$.
\end{proof}

\section{Selection problem}
In this section, under assumption (H3'), we investigate the selection criterion for weak solutions $(u^\lambda,m^\lambda)$ of $(qMFG_\lambda)$, as $\lambda$ tends to 0.
	
\begin{lem}\label{5 cm uni bound}
	The critical value $c(m)$ is uniformly bounded  on $\mathcal{P}(M)$.
\end{lem}
\begin{proof}
	Since for any $\varphi\in C^1(M)$,
	$$H(x,0,D\varphi)-\Vert F\Vert_\infty\le H(x,0,D\varphi)-F(x,m)\le H(x,0,D\varphi)+\Vert F\Vert_\infty,$$
	we have 
	$$\inf_{\varphi\in C^1(M)}\max_{x\in M}H(x,0,D\varphi)-\Vert F\Vert_\infty\le c(m)\le \inf_{\varphi\in C^1(M)}\max_{x\in M}H(x,0,D\varphi)+\Vert F\Vert_\infty.$$
	The proof is complete.
\end{proof}

\begin{lem}\label{5 critical value continuous}
	For any $m_1, m_2\in\mathcal{P}(M)$, we have 
	$$\vert c(m_1)-c(m_2)\vert\le C_Fd_1(m_1,m_2).$$
\end{lem}
\begin{proof}
	By the formula of Ma\~n\'e critical value, for any $m\in\mathcal{P}(M)$, 
	$$c(m)=\inf_{\varphi\in C^1(M)}\max_{x\in M}(H(x,0,D\varphi)-F(x,m)).$$
	Then we have for any $m_1,m_2\in\mathcal{P}(M)$,
	\begin{equation*}
		\begin{aligned}
			\vert c(m_1)-c(m_2)\vert\le\sup_{\varphi\in C^1(M)}\max_{x\in M}\vert F(x,m_1)-F(x,m_2)\vert\le C_F d_1(m_1,m_2).
		\end{aligned}
	\end{equation*}
\end{proof}

By Lemma \ref{5 cm uni bound} and Lemma \ref{5 critical value continuous}, the map $m\mapsto F(x,m)+c(m)$ is a nonlocal coupling term satisfying (F1)-(F3). We assume (H4) for the remainder of this section.
\begin{lem}\label{5 aubry set stable}
	Let $\{m_n\}_{n\in\mathbb{N}}\subset C([0,T];\mathcal{P}(M))$ be the sequence such that, as $n\to\infty$, $$\sup_{t\in[0,T]}d_1(m_n(t),m(t))\to 0,$$
	for some $m\in C([0,T];\mathcal{P}(M)).$ Then the sequence $\{x_{m_n}\}_{n\in\mathbb{N}}$ with $\mathcal{A}^{m_n(t)}=\{x_{m_n}\}$, for every $n\in\mathbb{N}$ and any $t\in[0,T]$, converges to a point $x_m$, and we have $\{x_m\}=\mathcal{A}^{m(t)}$ for any $t\in[0,T]$.
\end{lem}
\begin{proof}
	The proof is similar to the proof in \cite[Theorem 3.2]{CMM24arxiv}. By the definition of Aubry sets, for any fixed $t\in[0,T]$, we have
	\begin{equation*}
		\begin{aligned}
			&h_{m_n(t)}(x_{m_n},x_{m_n})\\
			=&\liminf_{\tau\to\infty}\left\{\inf_{\xi(0)=\xi(\tau)=x_{m_n}}\int_{0}^{\tau}L(\xi(s),0,\dot{\xi}(s))+F(\xi(s),m_n(t))ds+c(m_n(t))\tau\right\}\\
			=&0,
		\end{aligned}
	\end{equation*}
	where the second infimum is taken over all absolutely continuous curves. We choose a sequence $\{\tau_n\}$, where $\tau_n$ is sufficiently large, and a family of absolutely continuous curves $\{\gamma_n\}$, where $\gamma_n(0)=\gamma_n(\tau_n)=x_{m_n}$, such that, as $n\to\infty$, $\tau_n\to\infty$, $x_{m_n}\to \bar{x}$, and for every $n\in\mathbb{N}$, 
	$$\int_{0}^{\tau_n}L(\gamma_n(s),0,\dot{\gamma}_n(s))+F(\gamma_n(s),m_n(t))ds+c(m_n(t))\tau_n\le \dfrac{1}{n}.$$
	Then, by Azel\'a-Ascoli theorem and \cite[Proposition 3.1.4]{F08}, there exists an absolutely continuous curve $\bar{\gamma}$ such that $\gamma_n$ uniformly converges to $\bar{\gamma}$, and $\dot{\gamma}_n$ weakly converges to $\dot{\bar{\gamma}}$ in $\sigma(L^1,L^\infty)$ on any compact subsets of $[0,\infty)$. We define $d_n:=\vert x_n-\bar{x}\vert$ and a curve
	\begin{equation*}
		\tilde{\gamma}_n(s)=\left\{\begin{aligned}
			&\gamma_n^1=\dfrac{\bar{x}-x_n}{d_n}s+\bar{x},\quad &&s\in [-d_n,0),\\
			&\gamma_n(s),\quad &&s\in [0,\tau_n],\\
			&\gamma_n^2=\dfrac{x_n-\bar{x}}{d_n}(s-\tau_n-d_n)+x_n,\quad &&s\in (\tau_n,\tau_n+d_n].
		\end{aligned}\right.
	\end{equation*}
	It is clear that $\vert\dot{\gamma}_n^i\vert=1$ for $i=1,2$. Then we have 
	\begin{equation*}
		\begin{aligned}
			h_{m(t)}(\bar{x},\bar{x})&=\liminf_{\tau\to\infty}\left\{\inf_{\xi(0)=\xi(\tau)=\bar{x}}\int_{0}^{\tau}L(\xi(s),0,\dot{\xi}(s))+F(\xi(s),m(t))ds+c(m(t))\tau\right\}\\
			&\le \liminf_{\tau\to\infty}\left\{\int_{0}^{\tau}L(\bar{\gamma}(s),0,\dot{\bar{\gamma}}(s))+F(\bar{\gamma}(s),m(t))ds+c(m(t))\tau\right\}\\
			&\le\liminf_{n\to\infty}\left\{\int_{-d_n}^{\tau_n+d_n}L(\tilde{\gamma}_n(s),0,\dot{\tilde{\gamma}}_n(s))+F(\tilde{\gamma}_n(s),m(t))ds+c(m(t))(\tau+2d_n)\right\}\\
			&\le \liminf_{n\to\infty}\Bigg\{2d_n\left(\sup_{\substack{x\in M\\ \vert q\vert=1}}\vert L(x,0,q)\vert+\Vert F\Vert_\infty+\sup_{t\in[0,T]}c(m(t))\right)\\
			&\qquad\qquad+\int_{0}^{\tau_n}L(\gamma_n(s),0,\dot{\gamma}(s))+F(\gamma(s),m(t))ds+c(m(t))\tau\Bigg\}\\
			&\le \lim_{n\to\infty}\left\{2d_n\left(\sup_{\substack{x\in M\\ \vert q\vert=1}}\vert L(x,0,q)\vert+\Vert F\Vert_\infty+\sup_{t\in[0,T]}c(m(t))\right)+\dfrac{1}{n}\right\}=0,
		\end{aligned}
	\end{equation*}
	where the second inequality is due to the lower semicontinuity of action functional. Thus, for any $t\in[0,T]$, $\bar{x}\in\mathcal{A}^{m(t)}$, which implies that $\bar{x}=x_m$ by assumption (H4).
\end{proof}

The main issue on the selection problem is the regularity of $u(x,\cdot)$. 

\begin{prop}\label{5 regular of v lambda}
	For any $m\in C([0,T];\mathcal{P}(M))$ and $\lambda\in(0,1]$, let $u$ be the viscosity solution of
	\begin{equation}
		H(x,\lambda u,Du)=F(x,m(t))+c(m(t)),\quad x\in M,\ \forall t\in [0,T].
	\end{equation}
	Then there exists a modulus $\omega$, independent of $\lambda\in(0,1]$, such that for any $m\in C([0,T];\mathcal{P}(M))$, there holds that
\begin{equation*}
	\Vert (u(\cdot,t)-u(x_m,t))-(u(\cdot,s)-u(x_m,s))\Vert_\infty\le \omega(\vert t-s\vert),\quad \forall t,s\in[0,T].
\end{equation*}
\end{prop}

\begin{proof}
	By contradiction, we assume that there exists $\eps>0$ and some $m\in C([0,T];\mathcal{P}(M))$ such that there exists a sequence $\{m_n\}_{n\in\mathbb{N}}\subset C([0,T];\mathcal{P}(M))$ such that as $n\to\infty$, 
	$$\sup_{t\in[0,T]}d_1(m_n(t),m(t))\longrightarrow 0,$$
	a sequence $\{x_{m_n}\}_{n\in\mathbb{N}}$ such that $\{x_{m_n}\}=\mathcal{A}^{m_n}$ for every $n\in\mathbb{N}$, $\{x_{m_n}\}$ converges to the point $x_m$ with $\{x_m\}=\mathcal{A}^m$, a sequence $\{\lambda_n\}_{n\in\mathbb{N}}\subset(0,1]$, a sequence $\{t_n\}_{n\in\mathbb{N}}\subset[0,T-h_n]$ with $h_n\in(0,1/n)$ for every $n\in\mathbb{N}\setminus\{0\}$, such that
	$$\Vert(f_n(\cdot,t_n)-f_n(x_{m_n},t_n))-(g_n(\cdot,t_n+h_n)-g_n(x_{m_n},t_n+h_n))\Vert_\infty\ge\eps.$$
	where, for every $n\in \mathbb{N}$, $f_n(\cdot,t_n)$ is the viscosity solution of 
	$$H(x,\lambda_n f_n,Df_n)=F(x,m_n(t_n))+c(m_n(t_n)),\quad x\in M,$$
	and $g_n(\cdot,t_n+h_n)$ is the viscosity solution of
	$$H(x,\lambda_n g_n,Dg_n)=F(x,m_n(t_n+h_n))+c(m_n(t_n+h_n)),\quad x\in M.$$
	Then we take the subsequence of $\{\lambda_n\}$ to the limit, which converges to a constant $\lambda\in [0,1]$. There are two possible cases that either $\lambda=0$ or $\lambda\neq0$. We assume that $t_n$ and $t_n+h_n$ are both converges to $\tilde{t}$.
	
	By \cite[Lemma 2.3]{WYZ21}, $\{f_n(\cdot,t_n)\}$ and $\{g_n(\cdot,t_n+h_n)\}$ are uniformly bounded and equi-Lipschitz continuous on $M$ for all $\lambda\in(0,1]$, so we assume that there exist subsequences, still denoted by $\{f_n(\cdot,t_n)\}$ and $\{g_n(\cdot,t_n+h_n)\}$, uniformly converges to $f(\cdot,\tilde{t})$ and $g(\cdot,\tilde{t})$, respectively.
	
	For the case $\lambda\neq0$, since
	$$d_1(m_n(t_n),m(\tilde{t}))\le d_1(m_n(t_n),m_n(\tilde{t}))+d_1(m_n(\tilde{t}),m(\tilde{t}))\to 0,$$
	and assumption (F3), we have $F(\cdot,m_n(t_n))$ uniformly converges to $F(\cdot,m(\tilde{t}))$, and by Lemma \ref{5 critical value continuous}, $c(m_n(t_n))$ converges to $c(m(\tilde{t}))$. By the same argument, $F(\cdot,m_n(t_n+h_n))$ uniformly converges to $F(\cdot,m(\tilde{t}))$, and $c(m_n(t_n+h_n))$ converges to $c(m(\tilde{t}))$. By the stability of viscosity solutions, $f(\cdot,\tilde{t})$ and $g(\cdot,\tilde{t})$ are both viscosity solutions of
	$$H(x,\lambda u,Du)=F(x,m(\tilde{t}))+c(m(\tilde{t})),\quad x\in M,$$
	which implies that $f(\cdot,\tilde{t})=g(\cdot,\tilde{t})$ on $M$. Then, as $n\to\infty$, we have
	\begin{equation*}
		\begin{aligned}
			&\quad\ \  \Vert f_n(\cdot,t_n)-g_n(\cdot,t_n+h_n)\Vert_\infty\\
			&\le\Vert f_n(\cdot,t_n)-f(\cdot,\tilde{t})\Vert_\infty+\Vert g(\cdot,\tilde{t})-g_n(\cdot,t_n+h_n)\Vert_\infty
			\to0.
		\end{aligned}
	\end{equation*}
	Therefore, as $n\to\infty$,
	\begin{equation*}
		\begin{aligned}
			&\Vert (f_n(\cdot,t_n)-f_n(x_{m_n},t_n))-(g_n(\cdot,t_n+h_n)-g_n(x_{m_n},t_n+h_n))\Vert_\infty\\
			&\qquad\le2\Vert f_n(\cdot,t_n)-g_n(\cdot,t_n+h_n)\Vert_\infty\to0,
		\end{aligned}
	\end{equation*}
	which yields a contradiction.

	For the case $\lambda=0$, using the same argument as before, as $n\to\infty$, we have $x_{m_n}$ converges to $x_m$, $f_n(\cdot,t_n)$ uniformly converges to $f(\cdot,\tilde{t})$, $g_n(\cdot,t_n+h_n)$ uniformly converges to $g(\cdot,\tilde{t})$, and $f(\cdot,\tilde{t})$, $g(\cdot,\tilde{t})$ are both viscosity solutions of the equation
	\begin{equation}\label{5 H-J with u=0}
		H(x,0,Du)=F(x,m(\tilde{t}))+c(m(\tilde{t})),\quad x\in M.
	\end{equation}
	Therefore, $f(\cdot,\tilde{t})-f(x_m,\tilde{t})$ and $g(\cdot,\tilde{t})-g(x_m,\tilde{t})$ are also viscosity solutions of equation \eqref{5 H-J with u=0}, and they are equal on the projected Aubry set $\{x_m\}$, which implies that $$f(\cdot,\tilde{t})-f(x_m,\tilde{t})=g(\cdot,\tilde{t})-g(x_m,\tilde{t}).$$
	Finally, as $n\to\infty$, we obtain
	\begin{equation*}
		\begin{aligned}
			&\quad \Vert(f_n(\cdot,t_n)-f_n(x_{m_n},t_n))-(g_n(\cdot,t_n+h_n)-g_n(x_{m_n},t_n+h_n))\Vert_\infty\\
			&\qquad\le\Vert(f_n(\cdot,t_n)-f_n(x_{m_n},t_n)-(f(\cdot,\tilde{t})-f(x_m,\tilde{t}))\Vert_\infty\\
			&\qquad\quad+\Vert (g(\cdot,\tilde{t})-g(x_m,\tilde{t}))-(g_n(\cdot,t_n+h_n)-g_n(x_{m_n},t_n+h_n))\Vert_\infty\\
			&\qquad\qquad\longrightarrow\ 0,
		\end{aligned}
	\end{equation*}
	which yields a contradiction.
\end{proof}

\begin{cor}\label{5 u lam subconverge}
		For any $\lambda\in(0,1]$ and $m^\lambda\in C([0,T];\mathcal{P}(M))$, let $u^\lambda$ be the viscosity solution of the equation
	\begin{equation*}
		H(x,\lambda u^\lambda,Du^\lambda)=F(x,m^\lambda(t))+c(m^\lambda(t)),\quad x\in M,\ \forall t\in [0,T].
	\end{equation*}
	If $$\lim_{\lambda\to0}\sup_{t\in[0,T]}d_1(m^\lambda(t),m(t))=0,$$ for some $m\in C([0,T];\mathcal{P}(M))$, then, for any $\bar{x}\in M$, as $\lambda\to0$, the sequence $\{u^\lambda-u^\lambda(\bar{x},\cdot)\}$ uniformly converges to the function $\bar{v}-\bar{v}(\bar{x},\cdot)\in C(M\times[0,T])$, where $\bar{v}$ is a viscosity solution of 
	\begin{equation}\label{5 bar v vis sol}
		H(x,0,Dv)=F(x,m(t))+c(m(t)),\quad x\in M,\ \forall t\in[0,T].
	\end{equation}
	Moreover, $\bar{v}-\bar{v}(\bar{x},\cdot)=h_{m(\cdot)}(x_m,\cdot)-h_{m(\cdot)}(x_m,\bar{x})$.
\end{cor}
\begin{proof}
	By  Proposition \ref{5 regular of v lambda}, for any $t,s\in[0,T]$, we have
	\begin{equation*}
		\begin{aligned}
			&\bigg\Vert \bigg(u^\lambda(\cdot,t)-u^\lambda(\bar{x},t)\bigg)-\bigg(u^\lambda(\cdot,s)-u^\lambda(\bar{x},s)\bigg)\bigg\Vert_\infty\\
			\le&\bigg\Vert\bigg((u^\lambda(\cdot,t)-u^\lambda(x_m,t))+(u^\lambda(x_m,t)-u^\lambda(\bar{x},t))\bigg)\\
			&\qquad\qquad-\bigg((u^\lambda(\cdot,s)-u^\lambda(x_m,s))+(u^\lambda(x_m,s)-u^\lambda(\bar{x},s))\bigg)\bigg\Vert_\infty\\
			\le&\bigg\Vert \bigg((u^\lambda(\cdot,t)-u^\lambda(x_m,t))-(u^\lambda(\cdot,s)-u^\lambda(x_m,s))\bigg)\bigg\Vert_\infty\\
			&\qquad\qquad+\bigg\Vert\bigg((u^\lambda(x_m,t)-u^\lambda(\bar{x},t))-(u^\lambda(x_m,s)-u^\lambda(\bar{x},s))\bigg)\bigg\Vert_\infty\\
			\le &2\omega(\vert t-s\vert),
		\end{aligned}
	\end{equation*}
	where the modulus $\omega$ is the same as in Proposition \ref{5 regular of v lambda}. Thus, $\{u^\lambda-u^\lambda(\bar{x},\cdot)\}$ is uniformly bounded and equi-continuous.  With Azel\'a-Ascoli theorem, it converges (up to a subsequence) to some function $\bar{v}-\bar{v}(\bar{x},\cdot)$. Since the limit is independent of subsequences, the sequence$\{u^\lambda-u^\lambda(\bar{x},\cdot)\}$ uniformly converges to $\bar{v}-\bar{v}(\bar{x},\cdot)$. $\bar{v}$ is a viscosity solution of \eqref{5 bar v vis sol}, which is the consequence of \cite[Theorem 1.1]{WYZ21}.
	
	Recall that viscosity solutions of the equation \eqref{5 bar v vis sol} under assumption (H4) are unique up to additive constants. Since $\bar{v}-\bar{v}(\bar{x},\cdot)$ and $h_{m(\cdot)}(x_m,\cdot)-h_{m(\cdot)}(x_m,\bar{x})$ coincide on $\bar{x}$, they coincide everywhere. The proof is complete.
\end{proof}

\begin{proof}[Proof of Theorem \ref{main result 2}]
	For any fixed $\lambda>0$, there exists a weak solution of $(u^\lambda,m^\lambda)\in C(M\times[0,T])\times C([0,T];\mathcal{P}(M))$ of $(qMFG_\lambda)$. We first prove that for any $\lambda\in (0,1]$, $m^\lambda$ is uniformly bounded and equi-Lipschitz continuous. Since for any $t\in[0,T]$ and $\lambda\in(0,1]$, $u^\lambda(\cdot,t)$ are uniformly bounded and equi-Lipschitz continuous, by the arguments in Lemma \ref{4 flow Lip t}, Lemma \ref{4 flow inverse Lip x} and Proposition \ref{4 push-for ac bound}, we have $m^\lambda$ is uniformly bounded and equi-Lipschitz continuous. With Corollary \ref{5 u lam subconverge}, as $\lambda\to 0$, up to a subsequence, we have
	\begin{enumerate}[(c1)]
		\item There exists some $\bar{\mu}\in C([0,T];\mathcal{P}(M))$, such that $$\lim_{\lambda\to0}\sup_{t\in[0,T]}d(m^\lambda(t),\bar{\mu}(t))=0.$$
		\item $c(m^\lambda(t))$ converges to $c(\bar{\mu}(t))$ uniformly on $[0,T]$.
		\item For any $\bar{x}\in M$, there exists some function $\bar{v}\in C(M\times[0,T])$, such that $u^\lambda-u(\bar{x},\cdot)$ uniformly converges to $\bar{v}-\bar{v}(\bar{x},\cdot)$ on $M\times [0,T]$. Moreover, $$\bar{v}-\bar{v}(\bar{x},\cdot)=h_{\bar{\mu}(\cdot)}(x_m,\cdot)-h_{\bar{\mu}(\cdot)}(x_m,\bar{x}).$$
		\item $Du^\lambda$ converges to $D\bar{v}$ a.e. on $M\times[0,T]$.
		\item The pair $(\bar{v},\bar{\mu})$ is a weak solution of $(qMFG_0)$.
	\end{enumerate}
	(c1) is the consequence of aforementioned arguments. (c2) is the consequence of Lemma \ref{5 critical value continuous}. (c3) in the consequence of Corollary \ref{5 u lam subconverge}. (c4) is due to the uniform semi-concavity of $u^\lambda$.
	Subsequently, we prove (c5). By Corollary \ref{5 u lam subconverge}, $\bar{v}$ is a viscosity solution of 
	$$H(x,0,D\bar{v})=F(x,\bar{\mu}(t))+c(\bar{\mu}(t)),\quad x\in M,\ \forall t\in[0,T].$$
	 Then we prove $\bar{\mu}$ is the weak solution of
	 \begin{equation}\label{5 continuity eq 0}
	 	\left\{\begin{aligned}
	 		&\partial_t\bar{\mu}-\text{div}\left(\bar{\mu}\dfrac{\partial H}{\partial p}(x,0,D\bar{v})\right)=0,\quad (x,t)\in M\times(0,T],\\
	 		&\bar{\mu}(0)=m_0.
	 	\end{aligned}\right.
	 \end{equation}  For any $\lambda>0$, we have the pushforward $m^\lambda=\Phi^\lambda(\cdot,\cdot)\sharp m_0$ of the measure $m_0$, where $$\Phi^\lambda(x,t)=x-\int_{0}^{t}\dfrac{\partial H}{\partial p}(\Phi^\lambda(x,s),\lambda u^\lambda(\Phi^\lambda(x,s),s),Du^\lambda(\Phi^\lambda(x,s),s))ds.$$
	Define $\tilde{\mu}:=\overline{\Phi}(\cdot,\cdot)\sharp m_0$, where 
	$$\overline{\Phi}(x,t)=x-\int_{0}^{t}\dfrac{\partial H}{\partial p}(\overline{\Phi}(x,s),0,D\bar{v}(\overline{\Phi}(x,s),s))ds.$$
	Then for any function $f\in C(M)$ and any $t\in[0,T]$, we have
	\begin{equation*}
		\begin{aligned}
			\lim_{\lambda\to0}&\int_M f(x)dm^\lambda(t)=\lim_{\lambda\to0}\int_M f(\Phi^\lambda(x,t))dm_0\\
			=&\int_Mf(\overline{\Phi}(x,t))dm_0=\int_Mf(x)d\tilde{\mu}(t),
		\end{aligned}
	\end{equation*}
	which implies that $$\lim_{\lambda\to0}\sup_{t\in[0,T]}d_1(m^\lambda(t),\tilde{\mu}(t))=0.$$ The first and third equalities are due to the definition of pushforward, the second equality is due to (c4). Since $\bar{\mu}$ and $\tilde{\mu}$ are both the limit of the same subsequence of $\{m^\lambda\}$, $\bar{\mu}=\tilde{\mu}$ on $[0,T]$. It is clear that $\bar{\mu}(0)=m_0$, and $\bar{\mu}$ is the unique weak solution of \eqref{5 continuity eq 0} due to Proposition \ref{4 sol of continuity eq}.

	Thus, $(\bar{v},\bar{\mu})$ is a weak solution of $(qMFG_0)$, and the proof is complete.
\end{proof}

\medskip

{\bf There is no conflict of interest. There is no data in this paper.}

\medskip

\noindent {\bf Acknowledgements:}
Xiaotian Hu is supported by China Scholarship Council (Grant No.  202406230212). The author gratefully acknowledges Professors Annalisa Cesaroni and Cristian Mendico for their helpful discussions and constructive advice concerning this paper.

\bibliographystyle{plain}

\appendix
\begin{appendices}

\section{Existence result under (H3')}
In the Appendix, we show that the existence result for $(qMFG_\lambda)$ under (H1), (H2) and the weaker monotonicity assumption (H3'). For convenience, we only consider the case when $\lambda=1$. 

By \cite[Proposition A.1, Proposition B.1]{WWY19a}, for any $m\in\mathcal{P}(M)$, the Hamilton-Jacobi equation
\begin{equation}\label{App H-J}
	H(x,u,Du)=F(x,m)+c(m),\quad x\in M,
\end{equation}
admits a unique viscosity solution. We claim that $\{u_m\}_{m\in\mathcal{P}(M)}$ is uniformly bounded and equi-Lipschitz continuous.

\begin{lem}\label{App w equi Lip}
	The viscosity solution $w_m$ of the Hamilton-Jacobi equation
	$$H(x,0,Dw)=F(x,m)+c(m),\quad x\in M$$
	is equi-Lipschitz continuous with the Lipschitz constant $D_1$, where 
	$$D_1=\sup_{\substack{x\in M\\ \vert q\vert=1}}L(x,0,q)+\Vert F\Vert_\infty+\sup_{m\in\mathcal{P}(M)}c(m).$$
\end{lem}
\begin{proof}
	For any $x,y\in M$ and $t>0$, we define an absolutely continuous curve $\gamma:[0,t]\to M$ such that $\vert \dot{\gamma}\vert=1$, $\gamma(0)=x$ and $\gamma(t)=y$, then we have
	\begin{equation*}
		\begin{aligned}
			w_m(y)-w_m(x)&\le \int_{0}^{t}L(\gamma(s),0,\dot{\gamma}(s))+F(\gamma(s),m)ds+c(m)t\\
			&\le\int_{0}^{t}\sup_{\substack{x\in M\\ \vert q\vert=1}}L(x,0,q)+\Vert F\Vert_\infty ds+\sup_{m\in\mathcal{P}(M)}c(m)t\\
			&=\left(\sup_{\substack{x\in M\\ \vert q\vert=1}}L(x,0,q)+\Vert F\Vert_\infty+\sup_{m\in\mathcal{P}(M)}c(m)\right)\vert x-y\vert.
		\end{aligned}
	\end{equation*}
	The proof is complete.
\end{proof}
$w_m$ satisfies that for any $t\ge0$,
$$w_m=\inf_{\gamma(t)=x}\left\{w_m(\gamma(0))+\int_{0}^{t}L(\gamma(s),0,\dot{\gamma}(s))+F(\gamma(s),m)ds+c(m)t\right\}.$$

For any given $t>0$, $m\in\mathcal{P}(M)$ and $x,y\in M$, we define the minimal action by
$$h_t^m(x,y)=\inf_{\gamma}\int_{0}^{t}L(\gamma(s),0,\dot{\gamma}(s))+F(\gamma(s),m)ds,$$
where the infimum is taken among all absolutely continuous curves $\gamma$ such that $\gamma(0)=x$ and $\gamma(t)=y$.

\begin{prop}
	For any given $t_0>0$, there exists a constant $D_{t_0}$ such that 
	$$\vert h_t^m(x,y)+c(m)t\vert\le D_{t_0},\quad \forall x,y\in M,\ \forall m\in\mathcal{P}.$$
\end{prop}
\begin{proof}
	\textbf{Lower bound.} By \cite[Proposition 4.4.2]{F08}, for any $t>0$, $x,y\in M$, we have 
	$$w_m(x)-w_m(y)\le h_t^m(x,y)+c(m)t,\quad \forall m\in\mathcal{P}(M),$$ with Lemma \ref{App w equi Lip}, $w_m$ is equi-Lipschitz continuous. Then 
	$$w_m(x)-w_m(y)\ge -D_1\vert x-y\vert\ge -D_1\text{diam}(M),$$
	which implies that $h_t^m(x,y)+c(m)t$ has a lower bound, where $\text{diam}(M)$ denotes the diameter of $M$.
	
	\textbf{Upper bound.} By \cite[Proposition 4.4.4]{F08}, for any given $t_0>0$, there exists a constant $C_{t_0}$ independent of $m\in\mathcal{P}(M)$, such that for each $x,y\in M$, we can find a $C^\infty$ curve $\gamma:[0,t_0]\to M$ with $\gamma(0)=x$, $\gamma(t_0)=y$, and 
	$$\int_{0}^{t_0}L(\gamma(s),0,\dot{\gamma}(s))+F(\gamma(s),m)ds+c(m)t_0\le C_{t_0},$$
	where $C_{t_0}:=\tilde{C}_{t_0}$, $\tilde{C}_{t_0}=\sup_{\substack{x\in M\\ \vert q\vert\le \frac{\text{diam}(M)}{t_0}}}L(x,0,q)+\Vert F\Vert_\infty+\sup_{m\in\mathcal{P}(M)}c(m).$ Then by \cite[Lemma 5.3.2]{F08}, for any $t>t_0$, there exist a minimizer $\gamma_2:[t_0,t]\to M$ of $w_m(x)$ with $\gamma(t)=x$ and a $C^\infty$ curve $\gamma_1;[0,t_0]\to M$ with $\gamma_1(t_0)=\gamma_2(t_0)$, such that
	\begin{equation*}
		\begin{aligned}
			h_t^m(x,y)+c(m)t&\le \int_{0}^{t_0}L(\gamma_1(s),0,\dot{\gamma}_1(s))+F(\gamma_1(s),m)ds+c(m)t_0\\
			&\qquad \qquad+\int_{t_0}^{t}L(\gamma_2(s),0,\dot{\gamma}_2(s))+F(\gamma_2(s),m)ds+c(m)(t-t_0)\\
			&\le C_{t_0}+w_m(x)-w_m(\gamma_2(t_0))\\
			&\le C_{t_0}+D_1\text{diam}(M).
		\end{aligned}
	\end{equation*}
	Let $D_{t_0}:=\max\{ -D_1\text{diam}(M),C_{t_0}+D_1\text{diam}(M)\}$, the proof is complete.
\end{proof}

We define the semigroup for $L(x,u,q)$ to prove the uniform boundedness of $\{u_m\}$. For any continuous function $\phi\in C(M)$ and $t\ge0$, define 
\begin{equation*}
	\begin{aligned}
		T_t^m\phi(x):=\inf_{\gamma(t)=x}\left\{\phi(\gamma(0))+\int_{0}^{t}L(\gamma(s),T_s^m\phi(\gamma(s)),\dot{\gamma}(s)+F(\gamma(s),m))ds+c(m)t\right\},\\
	\end{aligned}
\end{equation*}
where the infimum is taken over all absolutely continuous curves $\gamma$ such that $\gamma(t)=x$. By \cite[Proposition 2.9, Proposition A.1]{WWY19a}, for any $m\in\mathcal{P}$, any $\phi\in C(M)$, the limit
$$u_m:=\lim_{t\to\infty}T_t^m\phi$$
is the unique viscosity solution of \eqref{App H-J}.

\begin{prop}\label{App semigroup bound}
	Given $t_0>0$, for any $\phi\in C(M)$, there exists a constant $D_2>0$, dependent on $t_0$ and $\phi$, such that 
	$$\Vert T^m_t\phi\Vert_\infty\le D_2,\quad \forall t\in[t_0,\infty),\ \forall m\in\mathcal{P}(M).$$
\end{prop}
\begin{proof}
	\textbf{Lower bound.} For any $(x,t)\in M\times[t_0,\infty)$, such that $T_t^m\phi(x)<0$, let $\alpha:[0,t]\to M$ be the minimizer of $T_t^m(x)$. We consider the function
	$$s\mapsto T_t^m(\alpha(s)).$$
	Since $T_0^m\phi(\alpha(0))=\phi(\alpha(0))$, $T_t^m\phi(x)<0$, there exists $s_0\in[0,t]$, such that $T_{s_0}^m\phi(\alpha(s_0))\ge\min\{\phi(\alpha(0)),0\}$ and $T_s^m\phi(\alpha(s))<0$, for $s\in(s_0,t]$. Then we have
	\begin{equation*}
		\begin{aligned}
			T_t^m\phi(x)&=T_{s_0}^m\phi(\alpha(s_0))+\int_{s_0}^{t}L(\alpha(s),T_s^m(\alpha(s)),\dot{\alpha}(s))+F(\alpha(s),m)ds+c(m)(t-s_0)\\
			&\ge\min\{\phi(\alpha(0)),0\}+\int_{s_0}^{t}L(\alpha(s),0,\dot{\alpha}(s))+F(\alpha(s),m)ds+c(m)(t-s_0)\\
			&\ge -\Vert \phi\Vert_\infty+h_{t-s_0}^m(\alpha(s_0),x)\\
			&\ge -\Vert \phi\Vert_\infty+D_{t_0},
		\end{aligned}
	\end{equation*}
	which gives a lower bound of $T_t\phi(x)$ on $M\times[t_0,\infty)$.
	
	\textbf{Upper bound}. For any $(x,t)\in M\times [t_0,\infty)$, such that $T_t^m\phi(x)>0$. Let $\beta:[0,t]\to M$ be a minimizer of $h_t^m(y,x)$ with $\beta(0)=y$ and $\beta(t)=x$. We consider the function
	$$s\mapsto T_s^m\phi(\beta(s)).$$
	Since $T_0^m\phi(\beta(0))=\phi(\beta(0))$, $T_t^m\phi(x)>0$, there exists $s_0\in[0,t]$, such that $T_{s_0}^m\phi(\beta(s_0))\le\max\{\phi(\beta(0)),0\}$ and $T_s^m\phi(\beta(s))>0$, for $s\in (s_0,t]$. Then we have
	\begin{equation*}
		\begin{aligned}
			T_t^m\phi(x)&\le T_{s_0}^m\phi(\beta(s_0))+\int_{s_0}^{t}L(\beta(s),T_s^m\phi(\beta(s)),\dot{\beta}(s))+F(\beta(s),m)ds+c(m)(t-s_0)\\
			&\le\max\{\Vert\phi\Vert_\infty,0\}+\int_{s_0}^{t}L(\beta(s),0,\dot{\beta}(s))+F(\beta(s),m)ds+c(m)(t-s_0)\\
			&\le \Vert \phi\Vert_\infty+h_{t-s_0}^m(\beta(s_0),x).
		\end{aligned}
	\end{equation*}
	It is clear that at least one of $s_0$ and $t-s_0$ is no less than $t_0/2$. If $s_0\ge t_0/2$, then $$h_{t-s_0}^m(\beta(s_0),x)=h_t^m(\beta(s_0),x)-h_{s_0}^m(\beta(0),\beta(s_0))\le 2D_{t_0/2}.$$
	If $t-s_0\ge t_0/2$, then $h_{t-s_0}^m(\beta(s_0),x)\le D_{t_0/2}$. We give the upper bound of $T_t^m(x)$ on $M\times[t_0,\infty)$.
	
	The proof is complete.
\end{proof}

By Proposition \ref{App semigroup bound}, since the unique viscosity solution of \eqref{App H-J} $u_m=\lim_{t\to\infty} T_t^m\phi$ for any $\phi\in C(M)$, the sequence $\{u_m\}_{m\in\mathcal{P}(M)}$ is uniformly bounded on $M$. We establish the equi-Lipschitz continuity of $u_m$ in the following proposition.

\begin{prop}
	$\{u_m\}_{m\in\mathcal{P}(M)}$ is equi-Lipschitz continuous on $M$.
\end{prop}
\begin{proof}
	For any $t>0$ and any $x,y\in M$, we define a curve $\gamma:[0,t]\to M$ with $\vert\dot{\gamma}\vert=1$, $\gamma(0)=x$ and $\gamma(t)=y$. Then we have
	\begin{equation*}
		\begin{aligned}
			u_m(y)-u_m(x)\le \int_{0}^{t}L(\gamma(s),u_m(\gamma(s)),\dot{\gamma}(s))+F(\gamma(s),m)ds+c(m)t\\
			\le \left(\sup_{\substack{x\in M\\ \Vert u_m\Vert_\infty\\ \vert q\vert=1}}L(x,u_m,q)+\Vert F\Vert_\infty+\sup_{m\in\mathcal{P}(M)} c(m)\right)\vert x-y\vert,
		\end{aligned}
	\end{equation*}
	By Proposition \ref{App semigroup bound} and Lemma \ref{5 cm uni bound}, the Lipschitz constant is uniformly bounded, which yields the result.
\end{proof}

To prove the existence result under assumption (H3'), we establish the continuity $t\mapsto u_{m(t)}(x)$, where $u_{m(t)}$ is the viscosity solution of 
\begin{equation}\label{App H-Jm}
	H(x,u,Du)=F(x,m(t))+c(m(t)),\quad x\in M,\ \forall t\in[0,T].
\end{equation}

\begin{prop}\label{AA continuity in t}
	For any $m\in C([0,T];\mathcal{P}(M))$, the viscosity solution $u_{m(\cdot)}$ of \eqref{App H-Jm} is continuous on $[0,T]$.
\end{prop}
\begin{proof}
	If not, there exist $\eps>0$, a sequence $t_n\in[0,T-h_n]$ with $h_n\in(0,1/n)$, for every $n\in\mathbb{N}\setminus\{0\}$, such that 
	$$\Vert f_n(\cdot,t_n)-g_n(\cdot,t_n+h_n)\Vert_\infty>\eps,$$
	where, for every $n\in \mathbb{N}$, $f_n(\cdot,t_n)$ is the viscosity solution of 
	$$H(x,f_n,Df_n)=F(x,m(t_n))+c(m(t_n)),\quad x\in M,$$
	and $g_n(\cdot,t_n+h_n)$ is the viscosity solution of
	$$H(x,g_n,Dg_n)=F(x,m(t_n+h_n))+c(m(t_n+h_n)),\quad x\in M.$$
	The remaining proof is the same as that in Proposition \ref{5 regular of v lambda}.
\end{proof}

The rest of the proof follows the same argument as in Section 4.

\section{Non-uniqueness}

In this part, we explain why it is particularly difficult to establish the uniqueness of weak solutions for $(qMFG)$ and $(qMFG_0)$.The uniqueness proof for the second order case (\cite[Theorem 2.5]{M20} and \cite[Theorem 3.4]{CM23}), strongly relies on the Lipschitz continuity of the map $t\mapsto Du(x,t)$, which is proved by the continuous dependence estimates based on the strong maximum principle and on the elliptic regularity. More precisely, there exists a constant $C\ge0$, such that for any $t,s\in[0,T]$, there exists
\begin{equation}\label{BB continuity of Du in t}
	\Vert Du(\cdot,t)-Du(\cdot,s)\Vert_\infty\le C\vert t-s\vert.
\end{equation}

However, in the first order case, the regularity of weak solutions is insufficiently to ensure this continuity, and the derivative $Du$ may exhibit singular points, where, for any fixed $t\in[0,T]$, a singular point of $u(\cdot,t)$ is a point $x_0$ such that $Du(x_0,t)$ does not exist. The singular point varies with $t$, which may lead to a jump in $Du$ between different singular points. We will present some counterexamples to illustrate the failure of the continuity of $Du$ in $t$. 

\begin{enumerate}[(Ex.1)]
	\item 
	We first consider an example from \cite{CZ17}, where the Hamiltonian is given by $H_\eps(x,p)=\dfrac{1}{2}\vert p\vert^2-2\sin^2(\pi x)+\eps p$, with the periodic boundary condition on $[0,1]$ and sufficiently small $\eps>0$. The corresponding Hamilton-Jacobi equation is $$\dfrac{1}{2}\vert Du\vert^2+\eps Du=2\sin^2(\pi x)-\dfrac{\eps^2}{2},\quad x\in[0,1].$$
	Let $u_0$ and $u_\eps$ denote the viscosity solutions in the case $\eps=0$ and $\eps>0$, respectively. It is straightforward to verify that the singular point of $u_0$ is at $x_0=\dfrac{1}{2}$, while the singular point of $u_\eps$ is at $x_\eps>\dfrac{1}{2}$, where $\cos(\pi x_\eps)=-\dfrac{\eps \pi}{4}$.
	Then we have that 
	\begin{equation*}
		Du_\eps(x)=\left\{\begin{aligned}
			&2\sin(\pi x)-\eps,\quad &&x\in[0,x_\eps),\\
			&-2\sin(\pi x)-\eps,\quad &&x\in(x_\eps,1],
		\end{aligned}\right.
	\end{equation*}
	and
	\begin{equation*}
		\Vert Du_\eps-Du_0\Vert_\infty\ge \sup_{\frac{1}{2}<x<x_\eps}\vert Du_\eps-Du_0\vert\ge \sup_{\frac{1}{2}<x<x_\eps}\vert 4\sin(\pi x)-\eps\vert.
	\end{equation*}
	If $0<\eps<\frac{3}{\pi}$, then $\Vert Du_\eps-Du_0\Vert_\infty=4\sqrt{1-\frac{\eps^2\pi^2}{16}}-\eps\ge \sqrt{7}-\frac{3}{\pi}>1.$
	
	\item Even though the projected Aubry set remains stable, singular points of viscosity solutions to the associated Hamilton-Jacobi equation may still vary. 
	
	First, we define a family of functions $\{g_n\}_{n\in\mathbb{N}}$ on $[n,n+1]$ with $g_0(x)=x^4(1-x)^4(1+x)$ on $[0,1]$. It is clear that $g_0^{(i)}(0)=g_0^{(i)}(1)=0$, for $i=0,1,2,3$. Then for each $n\in\mathbb{N}$, we set $g_n(x):=g_0(x-n)$ on $[n,n+1]$. Let $g(x):=g_n(x)$, which is a 1-periodic function of  class $C^3$.
	
	We consider the Hamiltonian
	$$H_\eps(x,p)=\dfrac{1}{2}p^2-2\sin^2(\pi x)-\eps g(x),$$
	for sufficiently small $\eps>0$, with the periodic boundary condition on $[0,1]$. Let $V_\eps(x):=-2\sin^2(\pi x)-\eps g(x).$ The critical value $$c(H_\eps)=\inf_{\varphi\in C^1(M)}\max_{x\in [0,1]}\left(\dfrac{1}{2}\vert D\varphi\vert^2+V_\eps(x)\right)=\max_{x\in [0,1]}V_\eps(x)=0.$$
	The projected Aubry set $\mathcal{A}_\eps=\{0\}$, and
	the derivative $Du_\eps$ is defined by
	\begin{equation*}
		Du_\eps(x)=\left\{\begin{aligned}
			&\sqrt{-2V_\eps(x)},\quad &&x\in[0,x_\eps),\\
			&-\sqrt{-2V_\eps(x)},\quad &&x\in(x_\eps,1].
		\end{aligned}\right.
	\end{equation*}
	It is clear that when $\eps=0$, the singular point is $x_0=\frac{1}{2}$. The following \ref{Fig.sub.1} shows the variation of singular points with $\eps$.
	\begin{figure}[H]
		\centering  
		\subfigure[{}]{
			\label{Fig.sub.1}
			\includegraphics[width=0.48\textwidth]{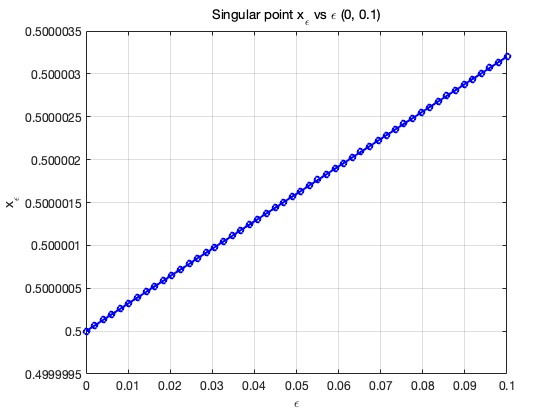}}
	\end{figure}
	We have 
	\begin{equation*}
		\Vert Du_{\eps_1}-Du_{\eps_2}\Vert_\infty\ge \max_{\frac{1}{2}<x\le\min\{x_{\eps_1},x_{\eps_2}\}}\left\vert \sqrt{-2V_{\eps_1}(x)}+\sqrt{-2V_{\eps_2}(x)}\right\vert>2.
	\end{equation*}
		
	An analogous behaviour can be observed in the contact case, where the Hamiltonians take the form $u + H_\epsilon(x, p).$
\end{enumerate}

From the above examples, we obverse that when the Hamilton-Jacobi equation is subject perturbations that cause the singular points of $u_\eps$ to vary, the derivative $Du_\eps$ fails to uniformly converge to $Du_0$, as $\eps\to0$, and a jump arises across the moving singular points.

For any $m\in C([0,T];\mathcal{P}(M))$ and $t,s\in[0,T]$, we interpret $d_1(m(t),m(s))$ as playing the role of the perturbation parameter $\eps$ in (Ex.1)-(Ex.2). This analogy suggests that the derivative $Du$ may exhibit discontinuity in $t$, and thus fails to be continuous on $[0,T]$, i.e. for any $t,s\in[0,T]$, the inequality \eqref{BB continuity of Du in t} does not hold.

\end{appendices}

\end{document}